\documentclass[11pt,a4paper]{article}
\pdfoutput=1

\usepackage[margin=28mm]{geometry}

\usepackage{mathtools,amssymb,amsthm}
\usepackage{microtype}
\usepackage{bm}
\usepackage{booktabs}
\usepackage{subcaption}
\usepackage{graphicx}
\usepackage{placeins}

\usepackage[T1]{fontenc}
\usepackage[utf8]{inputenc}
\usepackage{lmodern}

\usepackage[unicode,hidelinks]{hyperref}
\usepackage[nameinlink]{cleveref}

\hypersetup{
  pdftitle={Upper Bounds for Digitwise Generating Functions of Powers of Two: A Problem and a Matrix Representation},
  pdfauthor={Hideaki Noda},
  pdfkeywords={digit sum, powers of two, large deviations, generating function}
}

\newtheorem{theorem}{Theorem}
\newtheorem{lemma}{Lemma}
\newtheorem{proposition}{Proposition}

\newtheorem{problem}{Problem}
\newtheorem{remark}{Remark}

\theoremstyle{plain}

\crefname{theorem}{Theorem}{Theorems}
\Crefname{theorem}{Theorem}{Theorems}
\crefname{lemma}{Lemma}{Lemmas}
\Crefname{lemma}{Lemma}{Lemmas}
\crefname{proposition}{Proposition}{Propositions}
\Crefname{proposition}{Proposition}{Propositions}
\crefname{corollary}{Corollary}{Corollaries}
\Crefname{corollary}{Corollary}{Corollaries}
\crefname{remark}{Remark}{Remarks}
\Crefname{remark}{Remark}{Remarks}
\crefname{assumption}{Assumption}{Assumptions}
\Crefname{assumption}{Assumption}{Assumptions}
\crefname{problem}{Problem}{Problems}
\Crefname{problem}{Problem}{Problems}

\title{Upper Bounds for Digitwise Generating Functions of Powers of Two: A Problem and a Matrix Representation}
\author{Hideaki Noda}

\date{\today}

\begin{document}
\maketitle
\begin{abstract}
This short note studies the asymptotic behavior of a generating function associated with the decimal expansion of \(2^n\).
Our aims are twofold:
(i) to present a problem on the best possible upper bound for this behavior, and
(ii) to introduce a matrix representation that is useful for its analysis.
The representation corresponds to a finite-state transfer operator; analytic and dynamical aspects are not pursued here.
\end{abstract}

\section{Introduction}\label{sec:intro}
The decimal digits of powers of two have a simple yet rich structure.
In particular, the sequence of their last \(n\) decimal digits exhibits periodic and nontrivial patterns.
Studying these statistical properties has long been an interesting problem.
Such digit structures have been investigated in number theory and through harmonic-analytic methods, together with analytic tools related to digit distribution, generating functions, and exponential sums
(see, e.g., Stolarsky~\cite{Stolarsky1977,Stolarsky1978}, Stewart~\cite{Stewart1980}, Mauduit--Rivat~\cite{MauduitRivat2010}, Lagarias~\cite{Lagarias2009}, and references therein).

In this note, we consider the set \(\Omega_n\) of digit sequences corresponding to the last \(n\) decimal digits of powers of two.
For a nonnegative weight function \(h:\{0,1,\dots,9\}\to[0,\infty)\) that is not identically zero, we define the normalized log generating function over the last $n$ decimal digits by
\begin{equation}\label{eq:lgf}
\frac{1}{n}\log \Bigl(
\frac{1}{|\Omega_n|}
\sum_{\omega\in\Omega_n} \prod_{j=1}^n h(\omega_j)
\Bigr).
\end{equation}
(As shown in \Cref{sec:basic}, $|\Omega_n|=4\cdot5^{n-1}$.)
We seek an upper bound of this quantity in the asymptotic sense.
This quantity has a structure similar to that in large deviations theory, and it gives a natural measure for weighted statistics of digits.

To study this problem effectively, we introduce a sparse matrix representation that reflects the structure of the digit set \(\Omega_n\).
This representation provides a linear-algebraic viewpoint that is expected to be useful for discussing upper bounds.
It can also be viewed as a finite-state transfer operator on residue classes modulo \(2^n\),
capturing a weighted preimage structure.
However, the present note confines itself to combinatorial and matrix-theoretic aspects
and does not pursue analytic or dynamical interpretations.

The proposed framework can also be applied to the problem of linear growth of the decimal digit sum of \(2^n\).
Under certain natural assumptions, the framework implies such a linear growth, but the validity of these assumptions remains open.
This conditional result is recorded in Appendix~A as a supplement.

Finally, this short note is based on part of the results obtained in a longer manuscript in preparation~\cite{AuthorPrep}. That work focuses on the probabilistic analysis of the decimal digits of \(2^n\), while the present note concentrates on upper bounds and matrix representations, and gives an independent reconstruction from an algebraic point of view.

\section{Definitions and Basic Properties}\label{sec:basic}
Let $I=\{0,1,2,3,4,5,6,7,8,9\}$, and let $I^n$ be the $n$-fold Cartesian product of $I$.
For $x=(x_1,\dots,x_n)\in I^n$, define
\[
D_n(x)=\sum_{j=1}^n 10^{j-1} x_j .
\]
For example, $x=(7,0,2)$ yields $D_3(x)=207$.
For $n\ge 1$, define
\[
\Omega_n=\big \{ x\in I^n\ \big|\ \text{there exists } k\ge n \text{ such that } D_n(x)\equiv 2^k \pmod{10^n} \big \}.
\]
By Theorem~3.6 in \cite{nathanson2000}, the multiplicative order of $2$ modulo $5^L$ (that is, the smallest $k$ with $2^k\equiv 1 \pmod{5^L}$) is $4\cdot 5^{L-1}$.
Hence $|\Omega_n|=4\cdot 5^{n-1}$. Note that $2^k\bmod10$ cycles through $2,4,8,6$, so each $\omega\in\Omega_n$ has the last digit $\omega_1\in\{2,4,6,8\}$.
For a function $h:I\to\mathbb{C}$, set
\[
E=\sum_{k=0}^4 h(2k),\qquad O=\sum_{k=0}^4 h(2k+1).
\]
\medskip
We now define matrices and vectors.
For $n\ge 1$, define $M_n^{[h]} \in \mathbb{C}^{2^{n-1}\times 2^{n}}$ by the rule
\begin{equation}\label{eq:defMn}
(M_n^{[h]})_{i,t} =
\begin{cases}
h(j), & \text{if there exists } 0\le j\le 9 \text{ such that } t \equiv 10 i + j \pmod{2^n},\\[6pt]
0, & \text{otherwise},
\end{cases}
\end{equation}
for indices $0 \le i < 2^{n-1}$ and $0 \le t < 2^n$.
In other words, in row $i$, the nonzero entries are located at
\[
t \equiv 10 i + j \pmod{2^n}\qquad (0\le j\le 9),
\]
and the value at such a position is $h(j)$.
Thus $M_n^{[h]}$ is a sparse matrix that encodes the transition “append one digit $j$ to the state $i$ to form $10i+j$, and then reduce modulo $2^n$”, with weight $h(j)$.
We show a few examples for small $n$ to illustrate the structure more intuitively.
\[
M_1^{[h]}=\begin{bmatrix}
h(0){+}h(2){+}h(4){+}h(6){+}h(8) & h(1){+}h(3){+}h(5){+}h(7){+}h(9)
\end{bmatrix},
\]
\[
M_2^{[h]}=\begin{bmatrix}
h(0){+}h(4){+}h(8) & h(1){+}h(5){+}h(9) & h(2){+}h(6) & h(3){+}h(7)\\
h(2){+}h(6) & h(3){+}h(7) & h(0){+}h(4){+}h(8) & h(1){+}h(5){+}h(9)
\end{bmatrix},
\]
\[
\begingroup
\setlength{\arraycolsep}{1.6pt}
\renewcommand{\arraystretch}{0.95}
\scalebox{0.93}{$               
M_3^{[h]}=
\left[
\begin{array}{@{}*{8}{c}@{}}
h(0){+}h(8) & h(1){+}h(9) & h(2) & h(3) & h(4) & h(5) & h(6) & h(7)\\
h(6) & h(7) & h(0){+}h(8) & h(1){+}h(9) & h(2) & h(3) & h(4) & h(5)\\
h(4) & h(5) & h(6) & h(7) & h(0){+}h(8) & h(1){+}h(9) & h(2) & h(3)\\
h(2) & h(3) & h(4) & h(5) & h(6) & h(7) & h(0){+}h(8) & h(1){+}h(9)
\end{array}
\right],
$}
\endgroup
\]
\[
\begingroup
\setlength{\arraycolsep}{3.0pt}
\renewcommand{\arraystretch}{0.95}
\scalebox{0.93}{$               
M_4^{[h]} = \left[
\begin{array}{@{}*{16}{c}@{}}
h(0) & h(1) & h(2) & h(3) & h(4) & h(5) & h(6) & h(7) & h(8) & h(9) & 0 & 0 & 0 & 0 & 0 & 0 \\
h(6) & h(7) & h(8) & h(9) & 0 & 0 & 0 & 0 & 0 & 0 & h(0) & h(1) & h(2) & h(3) & h(4) & h(5) \\
0 & 0 & 0 & 0 & h(0) & h(1) & h(2) & h(3) & h(4) & h(5) & h(6) & h(7) & h(8) & h(9) & 0 & 0 \\
h(2) & h(3) & h(4) & h(5) & h(6) & h(7) & h(8) & h(9) & 0 & 0 & 0 & 0 & 0 & 0 & h(0) & h(1) \\
h(8) & h(9) & 0 & 0 & 0 & 0 & 0 & 0 & h(0) & h(1) & h(2) & h(3) & h(4) & h(5) & h(6) & h(7) \\
0 & 0 & h(0) & h(1) & h(2) & h(3) & h(4) & h(5) & h(6) & h(7) & h(8) & h(9) & 0 & 0 & 0 & 0 \\
h(4) & h(5) & h(6) & h(7) & h(8) & h(9) & 0 & 0 & 0 & 0 & 0 & 0 & h(0) & h(1) & h(2) & h(3) \\
0 & 0 & 0 & 0 & 0 & 0 & h(0) & h(1) & h(2) & h(3) & h(4) & h(5) & h(6) & h(7) & h(8) & h(9)
\end{array}
\right].
$}
\endgroup
\]

\begin{lemma}[Parity structure of column sums]\label{lem:colsum}
For $m\ge 1$ the matrix $M_m^{[h]}\in\mathbb{C}^{2^{m-1}\times 2^{m}}$ defined by \eqref{eq:defMn}, the sum of the entries in each column $t$ satisfies
\[
\sum_{i=0}^{2^{m-1}-1} (M_m^{[h]})_{i,t} =
\begin{cases}
E, & \text{if } t \text{ is even},\\[1ex]
O, & \text{if } t \text{ is odd}.
\end{cases}
\]
\end{lemma}

\begin{proof}
Fix $t$.
The condition $(M_m^{[h]})_{i,t}\neq 0$ holds if and only if there exists $j \in I$ such that $t \equiv 10i+j \pmod{2^m}$.
This is equivalent to
\[
5i \equiv \frac{t-j}{2} \pmod{2^{m-1}}.
\]
Since $\gcd(5,2^{m-1})=1$, there is a unique $i$ modulo $2^{m-1}$ if and only if the right-hand side is an integer, i.e., if and only if $t-j$ is even.
Hence the nonzero entries in column $t$ correspond exactly to the five digits $j$ that have the same parity as $t$, and their values are $h(j)$.
Therefore the column sum is $E$ when $t$ is even, and $O$ when $t$ is odd.
\end{proof}

\medskip
Let $m\ge 2$.
For each $i=1,2,3,4$, define
\begin{equation}\label{def_notation}
r_i^{(m)} = \Bigl\lfloor \frac{i 2^{m}}{5}\Bigr\rfloor,\qquad
u_i^{(m)} = (i 2^{m}) \bmod 5 \in \{1,2,3,4\},\qquad
\delta_i^{(m)} = 2 u_i^{(m)}.
\end{equation}
Then $u_i^{(m)}$ are pairwise distinct for $i=1,2,3,4$, and they take all values in $\{1,2,3,4\}$.
This matches the fact that $2^{m}\bmod 5$ cycles with period $4$ through $2,4,3,1$.
As a result,
\[
\bigl(\delta_1^{(m)},\delta_2^{(m)},\delta_3^{(m)},\delta_4^{(m)}\bigr)=
\begin{cases}
(2,4,6,8), & \text{if } m\equiv 0\pmod{4},\\
(4,8,2,6), & \text{if } m\equiv 1\pmod{4},\\
(8,6,4,2), & \text{if } m\equiv 2\pmod{4},\\
(6,2,8,4), & \text{if } m\equiv 3\pmod{4}.
\end{cases}
\]

\begin{lemma}[Four solutions to the compatibility congruence]\label{lem:B}
Let $0 \le r < 2^{m}$ and $\delta\in\{2,4,6,8\}$.
Consider the congruence
\[
10r+\delta\equiv 0 \pmod{2^{m+1}}.
\]
The set of all solutions $(r,\delta)$ is exactly $\{(r_i^{(m)},\delta_i^{(m)})\}_{i=1}^4$.
\end{lemma}

\begin{proof}
By definition we have $5 r_i^{(m)} + u_i^{(m)} = i 2^{m}$, so $(r_i^{(m)},\delta_i^{(m)})$ satisfies the congruence.
Fix $\delta\in\{2,4,6,8\}$.
The congruence $10r+\delta \equiv 0 \pmod{2^{m+1}}$ is equivalent to
\[
5r \equiv -\frac{\delta}{2} \pmod{2^{m}}.
\]
Since $\gcd(5,2^{m})=1$, this has a unique solution for $r$ modulo $2^{m}$.
For each $\delta$, this unique solution equals some $r_i^{(m)}$, and the corresponding $\delta_i^{(m)}$ equals $\delta$.
Hence the solution set is $\{(r_i^{(m)},\delta_i^{(m)})\}_{i=1}^4$.
\end{proof}

By Lemma~\ref{lem:B}, the values $\{\delta_i^{(m)}\}_{i=1}^4$ are exactly $\{2,4,6,8\}$ without repetition.
For $m\ge 2$ and a function $h:I\to\mathbb{C}$, define a column vector $b^{m,[h]}\in\mathbb{C}^{2^{m}}$ by
\[
b^{m,[h]}_r =
\begin{cases}
h(\delta_i^{(m)}), & \text{if } r = r_i^{(m)},\\[4pt]
0, & \text{otherwise}.
\end{cases}
\]

For $m\ge 1$ and functions $h_1,\dots,h_m:I\to\mathbb{C}$, define a row vector
$v^{m,[h_1,\dots,h_m]} \in \mathbb{C}^{1\times 2^{m}}$ by
\[
v_t^{m,[h_1,\dots,h_m]}
=\sum_{\substack{x\in I^m\\ D_m(x)\equiv t \  (\mathrm{mod}\ 2^m)}}
\ \prod_{j=1}^m h_j(x_j),\qquad 0\le t<2^m.
\]
In particular, when $h_1=\cdots=h_m=h$, we write $v_t^{m,[h]}$.
Equivalently, \(v^{m,[h_1,\dots,h_m]}\) may be viewed as a weighted residue-count vector modulo \(2^m\):  
each coordinate \(v_t^{m,[h_1,\dots,h_m]}\) represents the total weight of all \(m\)-digit decimal words \(x\in I^m\)  
satisfying \(D_m(x)\equiv t\pmod{2^m}\).

Finally, for a nonnegative function $h:I\to[0,\infty)$ that is not identically zero, define the normalized log generating function over the last $n$ decimal digits by
\[
  \Psi_n(h)
  = \frac{1}{n}\log \Bigl(
    \frac{1}{|\Omega_n|}
    \sum_{\omega\in\Omega_n}\prod_{j=1}^n h(\omega_j)
  \Bigr).
\]

\section{Matrix Representation}\label{sec:matrix}

In this section we give a matrix representation of $\sum_{\omega\in\Omega_n}\ \prod_{j=1}^n h_j(\omega_j)$.

\begin{lemma}[A recursion by adding the first digit]\label{lem:A}
Let $m\ge 2$. Let $h_1,\dots,h_m:I\to\mathbb{C}$. Then
\[
v^{m,[h_1,\dots,h_m]} = v^{m-1,[h_2,\dots,h_m]}  M_m^{[h_1]}.
\]
In particular, if all weights are the same function $h:I\to\mathbb{C}$, then
\[
v^{m,[h]} = v^{m-1,[h]}  M_m^{[h]}.
\]
For $m=1$ we have $v^{1,[h]} = M_1^{[h]}$.
\end{lemma}

\begin{proof}
Let $m\ge 2$.  
Take any $x=(x_1,\dots,x_m)\in I^m$. Write $x'=(x_2,\dots,x_m)\in I^{m-1}$ and $j=x_1\in I$, so $x=(j,x')$.  
Then $D_m(x)=10 D_{m-1}(x')+j$.
Let $0 \le t < 2^{m-1}$ be such that $t\equiv D_{m-1}(x')\pmod{2^{m-1}}$.  
Then $D_m(x)\equiv 10t+j \pmod{2^m}$.
Hence, for $0 \le s < 2^m$,
\begin{align*}
v^{m,[h_1,\dots,h_m]}_s
&=\sum_{j\in I}
\sum_{\substack{x'\in I^{m-1}\\ 10D_{m-1}(x')+j\equiv s\ (\mathrm{mod}\ 2^m)}}
h_1(j) \prod_{k=2}^m h_k(x_k) \\
&=\sum_{0\le t<2^{m-1}}
v_t^{m-1,[h_2,\dots,h_m]}\;
\sum_{\substack{j\in I\\ s\equiv 10t+j\ (\mathrm{mod}\ 2^m)}}
h_1(j).
\end{align*}
This is exactly the matrix product $v^{m-1,[h_2,\dots,h_m]} M_m^{[h_1]}$.  
Also, $v^{1,[h]}=M_1^{[h]}$ follows directly from the definition.
\end{proof}

\begin{proposition}[Matrix representation]\label{prop:main}
Let $n\ge 3$ and $h_1,\dots,h_n:I\to\mathbb{C}$. Then
\[
\sum_{\omega\in\Omega_n}\ \prod_{j=1}^n h_j(\omega_j)
=
M_1^{[h_n]}   M_2^{[h_{n-1}]} \cdots M_{n-1}^{[h_2]}   b^{n-1,[h_1]}.
\]
\end{proposition}

\begin{proof}
Let the left-hand side be $L$ and the right-hand side be $R$.  
For simplicity, write $r_i=r_i^{(n-1)}$ and $\delta_i=\delta_i^{(n-1)}$ as defined in \eqref{def_notation}.  
By Lemma~\ref{lem:A},
\[
M_1^{[h_n]} M_2^{[h_{n-1}]} \cdots M_{n-1}^{[h_2]}
= v^{n-1,[h_2,\dots,h_n]}.
\]
From the definitions of $v^{n-1,[h_2,\dots,h_n]}$ and $b^{n-1,[h_1]}$ we get
\begin{align*}
v^{n-1,[h_2,\dots,h_n]}  b^{n-1,[h_1]}
&= \sum_{i=1}^4 v_{r_i}^{n-1,[h_2,\dots,h_n]} h_1(\delta_i) \\
&= \sum_{i=1}^4 h_1(\delta_i)
   \sum_{\substack{x=(x_2,\dots,x_n)\in I^{n-1}\\ D_{n-1}(x)\equiv r_i\ (\mathrm{mod}\ 2^{n-1})}}
   \prod_{j=2}^n h_j(x_j).
\end{align*}

Now consider each inner sum.  
Let $\omega=(\delta_i,x_2,\dots,x_n)\in I^n$.  
By Lemma~\ref{lem:B},
\[
D_n(\omega)\equiv 10D_{n-1}(x)+\delta_i \equiv 10r_i+\delta_i \equiv 0 \pmod{2^n}.
\]
Here $\delta_i\in\{2,4,6,8\}$, so $\gcd(\delta_i,5)=1$.  
Thus $D_n(\omega)$ is coprime to $5$ and is an element of $(\mathbb{Z}/5^n\mathbb{Z})^\times$.  
Since $2$ is a primitive root modulo $5^n$ (see, e.g., Theorem 3.7 in \cite{nathanson2000}), there exists $m\in\mathbb{N}$ such that
$2^{m}\equiv D_n(\omega)\pmod{5^n}$.  
By adding a multiple of $\varphi(5^n)=4\cdot 5^{n-1}$ to the exponent, we can take $m_i\ge n$.  
Because $\gcd(5^n,2^n)=1$, we then have $D_n(\omega)\equiv 2^{m_i}\pmod{10^n}$.
Hence $\omega\in\Omega_n$, and so $L\ge R$.

Conversely, let $\omega=(\omega_1,\dots,\omega_n)\in\Omega_n$.  
Then $D_n(\omega)\equiv 0 \pmod{2^n}$, so $\omega_1\in\{2,4,6,8\}$.  
By Lemma~\ref{lem:B}, the congruence $10t+\omega_1\equiv 0 \pmod{2^n}$ has a unique solution $t=r_i$ for some $i$.  
Let $\tilde{x}=(\omega_2,\dots,\omega_n)$. Then $D_{n-1}(\tilde{x})\equiv r_i \pmod{2^{n-1}}$.  
Thus this $\omega$ appears exactly once in the $i$-th term of the double sum above.  
Therefore $L\le R$, and we conclude $L=R$.
\end{proof}

\begin{remark}[Finite-state preimage transfer operator]
For each $m$, the matrix $M_m^{[h]}$ defines a finite-state preimage transfer operator
\(T_m^{[h]} \colon \mathbb{C}^{2^{m}}\to\mathbb{C}^{2^{m-1}}\) on functions
\(f:\mathbb{Z}/2^{m}\mathbb{Z}\to\mathbb{C}\), defined by
\[
(T_m^{[h]} f)(i)=\sum_{j=0}^9 h(j)\, f\bigl((10i+j)\bmod 2^{m}\bigr),
\]
where \(i\in \mathbb{Z}/2^{m-1}\mathbb{Z}\).
In terms of matrices, \(T_m^{[h]}\) corresponds to the transpose of \(M_m^{[h]}\),
since the main text adopts the row-vector convention.
Analytic and dynamical aspects are not pursued further in this note.
\end{remark}

\section{\texorpdfstring{An Asymptotic Upper Bound for $\Psi_n(h)$}{An Asymptotic Upper Bound for Psi\_n(h)}}\label{sec:ldp}
In this section we study an asymptotic upper bound for $\Psi_n(h)$.
Throughout this section, unless otherwise stated, we consider a single nonnegative weight function $h:I\to[0,\infty)$ that is not identically zero.  
To simplify notation, we write $v^m = v^{m,[h]}$ and $M_m = M_m^{[h]}$ whenever no confusion arises.

\begin{proposition}[General upper bound using the 1-norm]\label{prop:1norm}
For any $h:I\to[0,\infty)$, we have
\[
\limsup_{n\to\infty}\Psi_n(h)
\ \le\
\log \Bigl(\frac{\max\{E,O\}}{5}\Bigr).
\]
\end{proposition}

\begin{proof}
In this proof, the notation $\|\cdot\|_1$ denotes the \emph{induced matrix $1$-norm},
that is, the maximum column sum norm.
When a row or column vector is regarded as a $1\times N$ or $N\times1$ matrix,
the same convention applies; in particular, for a scalar $a\ge0$ we have $\|a\|_1 = a$.
Let $S = \max\{E,O\}$.  
By Lemma~\ref{lem:colsum}, $\|M_m\|_1 = S$ for every $m$.  
Using the submultiplicativity of the $1$-norm, we obtain
\[
\|v^m\|_1 = \|v^{m-1}M_m\|_1
 \le \|v^{m-1}\|_1 \|M_m\|_1
 \le \cdots \le \|v^1\|_1 S^{m-1}.
\]
Since $v^1 = M_1 = [E\ \ O]$, we have $\|v^1\|_1 = \max\{E,O\} = S$.  
Therefore,
\[
\|v^{n-1}\|_1 \le S^{n-1}.
\]
Note that, by the definition of $b^{n-1}$, we have $\|b^{n-1}\|_1 = h(2)+h(4)+h(6)+h(8)$.  
From Proposition~\ref{prop:main}, we obtain
\[
\sum_{\omega\in\Omega_n}\prod_{j=1}^n h(\omega_j)
=\| v^{n-1}b^{n-1} \|_1
\le 
\|v^{n-1}\|_1 \|b^{n-1}\|_1
 \le
(h(2)+h(4)+h(6)+h(8)) S^{n-1}.
\]
This inequality immediately implies the claim.
\end{proof}

\begin{proposition}[The case $E=O$]\label{prop:equalEO}
Assume that the weight function $h:I\to[0,\infty)$ satisfies $E=O$.  
Then
\[
\lim_{n\to\infty}\Psi_n(h)
=\log \Big( \frac{E}{5} \Big) = \log \Big(\frac{1}{10}\sum_{j=0}^{9}h(j)\Big).
\]
\end{proposition}

\begin{proof}
By Lemma~\ref{lem:colsum}, each matrix $M_n$ has all column sums equal to $E$.  
The vector $v^1$ has all entries equal to $E$, and multiplying a uniform vector by $M_{m+1}$ again gives a uniform vector.  
Hence, by Lemma~\ref{lem:A}, we have
\[
v^m = (E^m,\dots,E^m)\in\mathbb{R}^{1\times 2^m}\quad (m\ge1).
\]
For $n\ge2$, Proposition~\ref{prop:main} and Lemma~\ref{lem:B} give
\[
\sum_{\omega\in\Omega_n}\prod_{j=1}^n h(\omega_j)
= v^{n-1}b^{n-1}
= E^{n-1}(h(2)+h(4)+h(6)+h(8)).
\]
The claim follows immediately.
\end{proof}

This corresponds to the case where the weights assigned to even and odd digits are completely uniform.  
It shows that the asymptotic behavior of the generating function can be described by the logarithm of a simple average value.

Here a natural question arises.  
If we assume that the digit sequence is sufficiently random,  
then the corresponding logarithmic average growth rate is expected to be bounded above by $\log \big(10^{-1}\sum_{j=0}^{9}h(j)\big)$.
This idea starts from the behavior of generating functions under a uniform distribution,  
and naturally generalizes the expectation that such statistical properties approximately extend to the set $\Omega_n$ as well.

The main problem of this note is stated as follows.

\begin{problem}[Best Possible Upper Bound]\label{prob:general}
For a general nonnegative weight function $h$, determine the best possible upper bound.  
We ask whether the following inequality holds:
\[
\limsup_{n\to\infty}\Psi_n(h)
\le
\log \Big(\frac{1}{10}\sum_{j=0}^{9}h(j)\Big).
\]
\end{problem}

\appendix
\section{Conditional Implication under an Upper-Bound Assumption}\label{sec:appendix-linear}

In this appendix, we record a conditional result that follows if a certain natural upper-bound assumption holds for the problem stated in the main text.  
Under this assumption, the digit sum of $2^n$ grows linearly in $n$.  
Although this result is not the main topic of the note, it may serve as an example of how the matrix framework introduced earlier can be applied.
In this appendix,
we fix the weight function by setting \(h(0)=x\ge 1\) and \(h(k)=1\) for \(k\neq 0\). For $\omega\in\Omega_n$, define
\[
N_0(\omega)=\#\{1\le j\le n:\ \omega_j=0\}.
\]
Then we have
\[
\Psi_n(h)
= \frac{1}{n}\log\Bigl(\frac{1}{|\Omega_n|}
\sum_{\omega\in\Omega_n} x^{N_0(\omega)} \Bigr)
= \frac{1}{n}\log\Bigl(\frac{1}{|\Omega_n|} M_1^{[h]}   M_2^{[h]} \cdots M_{n-1}^{[h]}   b^{n-1,[h]} \Bigr).
\]
This is a special case of Problem~\ref{prob:general}, but we restate it as a separate problem.

\begin{problem}[Upper bound for the normalized log generating function]\label{prob:mgf}
We ask whether, for every $x \ge 1$, the following inequality holds:
\begin{equation}\label{eq:premise-limsup}
\limsup_{n\to\infty} \Psi_n(h)
\le \log\Bigl(\frac{x+9}{10}\Bigr).
\end{equation}
\end{problem}

The inequality \eqref{eq:premise-limsup} is a very strong statement.  
Assuming that it holds, we can apply the standard exponential Chebyshev (Chernoff) estimate~\cite[§2.2]{dembo1998}  
to obtain the following result.

\begin{theorem}[Bounded zero proportion under the upper-bound assumption]\label{thm:zero_num}
Assume that Problem~\ref{prob:mgf} holds.  
Then there exist constants $c>0$ and $M\in\mathbb{N}$ such that,  
for every $n \ge M$ and every $\omega \in \Omega_n$,
\[
N_0(\omega) < c n.
\]
\end{theorem}

\begin{proof}
For $0<a<1$ and $x\ge1$, the Markov inequality gives
\[
\#\{\omega: N_0(\omega)\ge an\}
\le
\sum_{\omega\in\Omega_n} \mathbf{1}_{\{N_0\ge an\}}
\le
x^{-an}\sum_{\omega\in\Omega_n}x^{N_0(\omega)}.
\]
From \eqref{eq:premise-limsup} we obtain
\[
\limsup_{n\to\infty}\frac{1}{n}\log\#\{N_0\ge an\}
\le
\log5 - a\log x + \log\Bigl(\frac{x+9}{10}\Bigr).
\]
Setting $t=\log x$ and optimizing the right-hand side gives
\begin{equation}\label{eq:count-limsup}
\limsup_{n\to\infty}\frac{1}{n}\log\#\{\omega\in\Omega_n: N_0(\omega)\ge an\}
\le
\log5 - I(a),
\end{equation}
where
\[
I(a)=\sup_{t\ge0}\Bigl\{a t - \log\frac{e^{t}+9}{10}\Bigr\}
=
\begin{cases}
0, & 0\le a\le \displaystyle \frac{1}{10},\\[1ex]
\displaystyle a\log(10a)+(1-a)\log \Bigl(\frac{10(1-a)}{9}\Bigr), & \displaystyle \frac{1}{10}\le a<1.
\end{cases}
\]
The function $I(a)$ is continuous and increasing, and satisfies $I(1/10)=0$.  
It coincides with the upper-tail rate function in Cramér’s large deviations principle  
for i.i.d. Bernoulli trials with success probability $p=1/10$.  
There exists a unique $a_0 \approx 0.8649 \in(0,1)$ satisfying $I(a_0)=\log5$.  
For any $c>a_0$, we have $\log5-I(c)<0$.  
Hence, by \eqref{eq:count-limsup}, there exists $M=M(c)$ such that for all $n\ge M$,
\[
\#\{\omega: N_0(\omega)\ge c n\}=0.
\]
Thus, for every $\omega\in\Omega_n$, we have $N_0(\omega)<c n$.
\end{proof}

As a consequence, if Problem~\ref{prob:mgf} holds,  
the decimal digit sum $s(2^n)$ increases linearly in $n$, 
where $s(m)$ denotes the sum of the decimal digits of $m$. 
Indeed, by Theorem~\ref{thm:zero_num}, the proportion of nonzero digits among the last $n$ decimal digits is at least $(1-c)$,  
and hence for $n\ge M$ we have $s(2^n)\ge (1-c)n$.

To obtain linear growth of the digit sum, one may consider different forms of the upper bound on the right-hand side of \eqref{eq:premise-limsup}.
The function $\log((x+9)/10)$ used above is just one example.
The main idea is that there is an envelope function $g(x)$ for $\Psi_n(h)$,
and we can apply a Chernoff-type estimate to it.
For example, the same argument works if we relax the condition and use functions such as
\[
\log\frac{x^{1.1}+9}{9}, \qquad f_3(x)=\log\frac{x^{7.2}+3.224\times10^{9}}{1.157\times10^{9}}.
\]
Here the first expression is given purely as an illustration and will not be used later.
The key point in the Chernoff estimate is that we still get
exponential decay for the number of zeros.
Figure~\ref{fig:premise-check}(\subref{fig:premise-check-b})
 shows the proven upper bound $f_2(x)=\log\bigl((x+4)/5\bigr)$ from Proposition \ref{prop:1norm},
the expected upper bound $f_1(x)=\log\bigl((x+9)/10\bigr)$, and the relaxed example $f_3(x)$.
As shown in Figure~\ref{fig:premise-check}(\subref{fig:premise-check-b}), $A=\{x\in[1,\infty)\mid f_3(x)\le \log((x+4)/5)\}$ is a finite interval.
Here $f_3$ is used only as an auxiliary comparison function;
its constants have no special meaning and are simply chosen so that
$f_1(x)\le f_3(x)$ for all $x>1$, and $f_3(x)\le f_2(x)$ for all $x\in A$.
We now state the following related problem.

\begin{problem}[Relaxed upper bound for the normalized log generating function]\label{prob:limit_interval}
We ask whether, for any $x \in A$, the following inequality holds:
\begin{equation}\label{eq:premise-limsup2}
\limsup_{n\to\infty}\frac{1}{n}\log\Bigl(\frac{1}{|\Omega_n|}
\sum_{\omega\in\Omega_n} x^{N_0(\omega)} \Bigr)
\le f_3(x).
\end{equation}
\end{problem}

If this problem is true, then by Proposition~\ref{prop:1norm} the same bound \eqref{eq:premise-limsup2}  
holds for all $x\ge1$, and the same argument as above  
leads again to the linear growth of the digit sum of $2^n$.

\section{Numerical Computation}

\begin{figure}[t]
  \centering
  \begin{subfigure}[t]{0.48\textwidth}
    \centering
    \includegraphics[width=\textwidth]{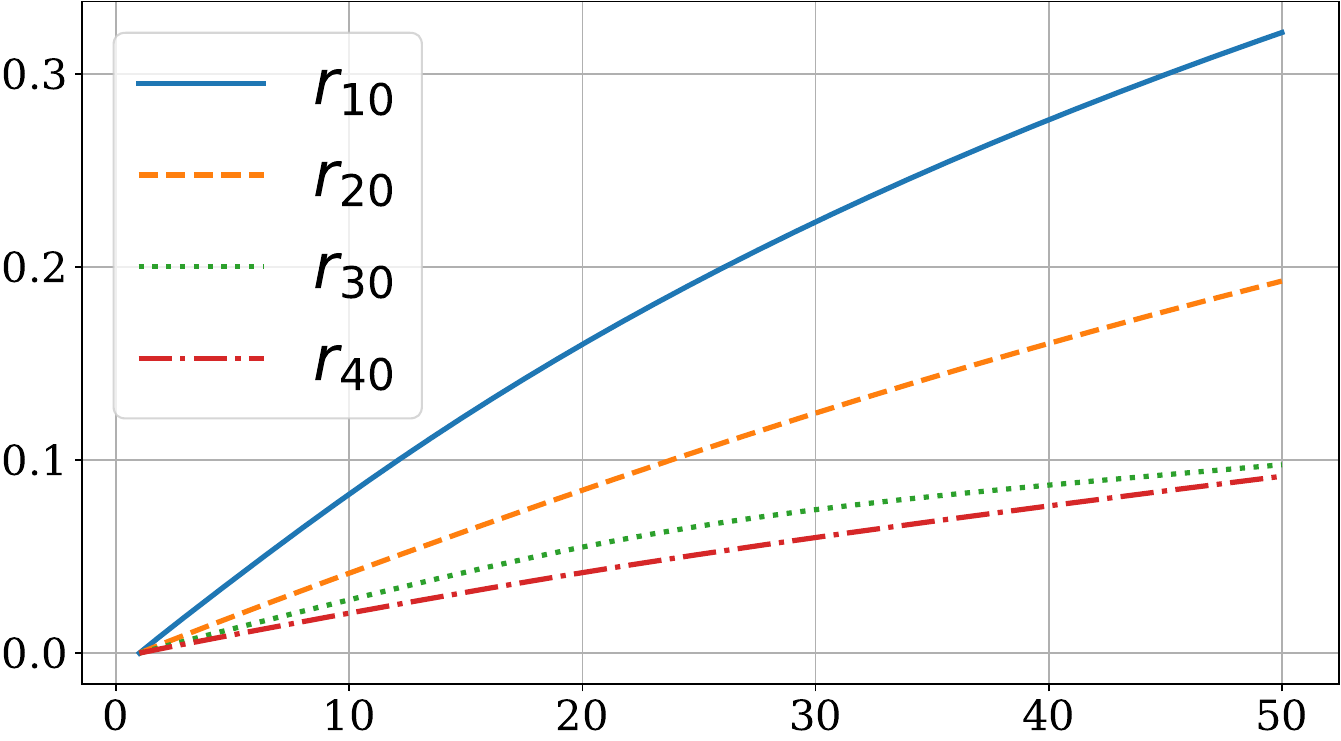}
    \caption{Difference from the expected bound: $r_n(x)$ for $n=10,20,30,40$}
    \label{fig:premise-check-a}
  \end{subfigure}
  \hfill
  \begin{subfigure}[t]{0.48\textwidth}
    \centering
    \includegraphics[width=\textwidth]{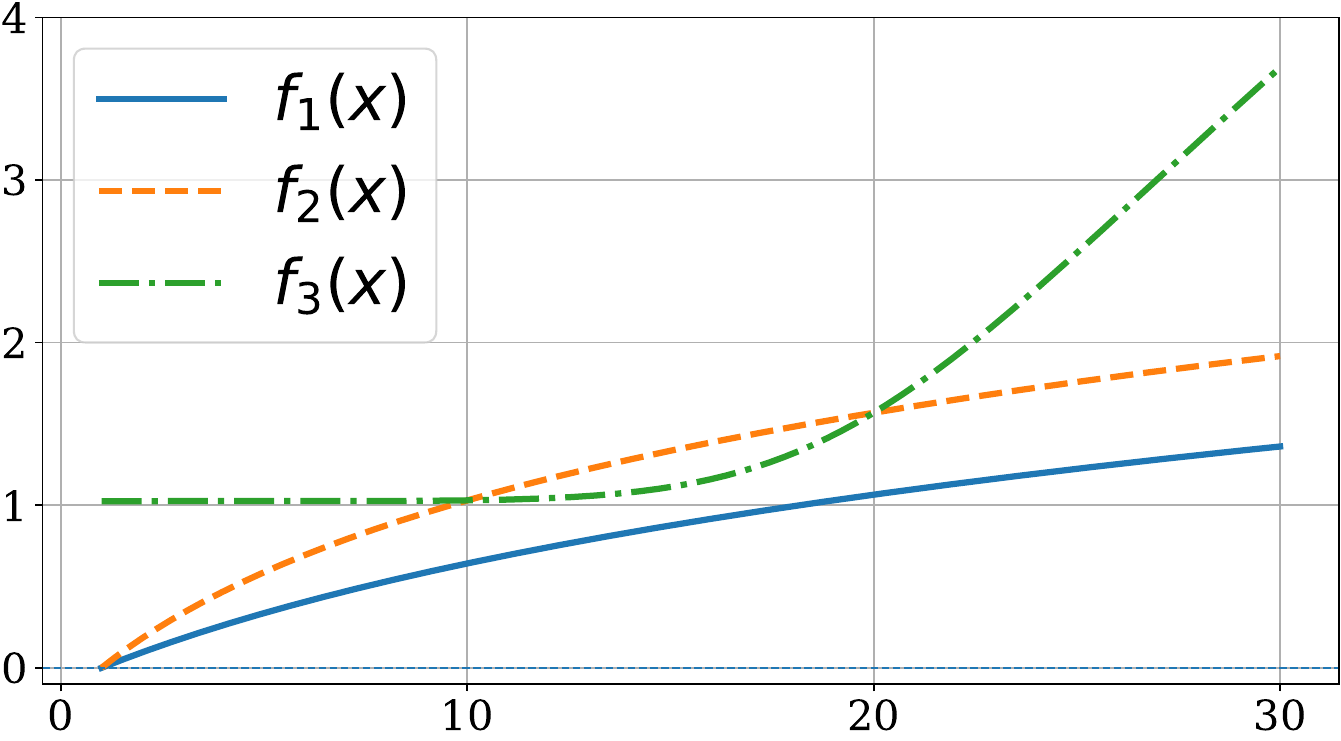}
    \caption{Three candidate upper bounds:
      \(f_1(x)\) (expected),
      \(f_2(x)\) (proved; Proposition~\ref{prop:1norm}),
      and \(f_3(x)\) (relaxed).}
    \label{fig:premise-check-b}
  \end{subfigure}
  \caption{Numerical verification of the validity of Problem~\ref{prob:mgf} and comparison of three possible upper-bound functions.}
  \label{fig:premise-check}
\end{figure}

We perform numerical verification for the validity of \cref{prob:mgf}.  
For $x > 1$, define
\[
r_n(x)
= \log\Bigl(\frac{x+9}{10}\Bigr)
- \frac{1}{n}\log\Bigl(\frac{1}{|\Omega_n|}
\sum_{\omega\in\Omega_n} x^{N_0(\omega)} \Bigr).
\]
Figure~\ref{fig:premise-check}(\subref{fig:premise-check-a})
 shows the numerical results for $n = 10,\, 20,\, 30$ and $40$.  
Since the computational cost of the second term grows rapidly with $n$,  
the computation is limited to relatively small values of $n$.  
Therefore, these results do not provide strong evidence for the validity of \cref{prob:mgf}.  
Within the computable range, we did not observe any violation of the upper bound.

\FloatBarrier

\bibliographystyle{plainurl}

\end{document}